\documentclass[11pt,reqno]{amsart}

\usepackage{amsmath,amssymb}
\usepackage{latexsym}
\usepackage{epsfig}
\usepackage{cite}
\usepackage[english]{babel}
\usepackage{booktabs}
\usepackage{graphicx}
\usepackage{tikz}
\usepackage{pgfplots}
\usepackage{eurosym}
\usepackage[inner=4cm,outer=2cm]{geometry}
\usepackage{amsthm}

\usetikzlibrary{arrows,decorations.pathmorphing,backgrounds,positioning,fit,petri,decorations}
\usetikzlibrary{calc,intersections,through,backgrounds,mindmap,patterns,fadings}
\usetikzlibrary{decorations.text}
\usetikzlibrary{decorations.fractals}
\usetikzlibrary{fadings}
\usepackage{pgflibraryshapes}
\usetikzlibrary{lindenmayersystems}
\usetikzlibrary{shadings}
\usetikzlibrary{shadows}
\usetikzlibrary{shapes.geometric}
\usetikzlibrary{shapes.callouts}
\usetikzlibrary{shapes.misc}
\usetikzlibrary{spy}
\usetikzlibrary{topaths}
\usetikzlibrary{datavisualization}
\usetikzlibrary{datavisualization.formats.functions}

\newtheorem{teo}{Theorem}
\newtheorem{rem}{Remark}
\newtheorem{prop}{Proposition}
\newtheorem{tdef}{Definition}

\numberwithin{equation}{section}

\numberwithin{equation}{section}

\title{Concentration of measure for classical Lie groups}
\author{S.L.~Cacciatori$^{1,2}$ and P.~Ursino$^{3}$}

\address{$^1$ Department of Science and High Technology, Universit\`a dell'Insubria, Via Valleggio 11, IT-22100 Como, Italy}
\address{$^2$ INFN sezione di Milano, via Celoria 16, IT-20133 Milano, Italy}
\address{$^3$ Department of Mathematics and Informatics, Universit\`a degli Studi di Catania, Viale Andrea Doria 6, 95125 Catania, Italy}

\begin{document}
\footnote{The second author gratefully acknowledges partial support from the projects MEGABIT -- Universit\`{a} degli Studi di Catania, PIAno di inCEntivi per la RIcerca di Ateneo 2020/2022 (PIACERI), Linea di intervento 2}
\maketitle
\begin{abstract}

We study concentration of measure in Lie group actions.
We define the notion of concentration locus of a flag sequence of Lie groups.
Some examples of infinite group action on an infinite dimensional compact and non compact manifold show the role played by the trajectory of concentration locus.
We also provide some applications in Physics, which emphasize the role of concentration of measure in gravitational effects.

\end{abstract}

\section*{Introduction}

\noindent
 Our work has two primary goals.
The first consists of defining a tool, the \textit{concentration loci}, which allow to determine whether an infinite sequence of open sets yields a concentration of measure into a significantly small object, then in Proposition \ref{ConLocConcSpace} we determine a connection between the notion of concentration locus and the notion of concentration in the sense of Gromov \cite{Gro99}.
We then study the concentration inherited by Lie group actions on compact and non compact infinite dimensional spaces.

A second goal consists of investigating the geometric trajectory of a concentration locus while concentrating the measure by increasing the dimension.
The concentration of measure, especially in relation to an increase of  dimensions, has an interesting physical counterpart, which we shall discuss in two physical applications.

In Physics, after Einstein, spacetime is described by a four dimensional manifold with a dynamical Lorentzian geometry.
However, at present there is no reason to think that such a description should continue to work at sub-Planckian scales.
On the contrary, quantum effects are expected to become dominant at these scales, even though any theory proposed so far is mainly at a speculative level.
Quantum corrections are expected, for example, to solve the problem of curvature singularities in black hole physics, as predicted by singularity theorems.
There is no proof indicating exactly what happens below such scales.
For example, regarding curvature singularities of black holes, we may conjecture that they consist of an infinite energy density concentrating into a point.
Remaining at a speculative level, it may be useful to think of the dynamical process of such a collapse: for instance, one cannot exclude that as soon as in a very small region the energy density grows excessively, the number of spatial dimensions
increases more and more.
This would in turn lead to a ``spreading of directions'', yield a lowering of the energy density, and, maybe, smooth out the singularity, since, as it is well known, the Lebesgue measure concentrates on the boundaries when dimensions increase.\\
Further, the concentration of energy in higher dimensions could lead to new kind of forces induced from the ``infinite'' dimensional spacetime to the usually visible finite dimensional world, and could then appear as a mysterious force in lower dimensions.
We conjecture that one of these forces is indeed due to the phenomenon of measure concentration.
Of course, all our arguments are merely speculative, and to substantiate them we need to improve our understanding of the phenomenon of measure concentration, which does not depend only on the varying of dimensions but also on topology.

In Mathematics, a first illustration of the concentration of measure is an elementary calculus performed on spheres by Levy \cite{Le}, where it is shown that by increasing the dimensions of spheres the measure concentrates along the equator, an
$(n-1)$-sphere.
The work of Gromov and Milman \cite{GrMi} has given a strong impetus to deepen this area of research.
The reason must be sought in the fact that, by isoperimetric inequalities, they showed that this phenomenon is uniform for compact manifolds and depends only on the rate of curvature.
They also exhibit, using fixed point arguments, examples of some natural fibrations, which have continuous sections but have no uniformly continuous sections (see also \cite{RU20} for connections between extreme amenability and the existence of
global continuous sections).

One of the most important applications of concentration of measure is the existence of fixed points on compacta. A topological group
$G$ has a very strong fixed point property (i.e., Extreme Amenability), if any continuous $G$-action on a compact set has a fixed point (i.e.,
there exists a point $x$ such that $g\cdot x = x,\ \forall g\in G$).

Extreme amenability of automorphism groups of countable structures has been related with Ramsey property, a pure combinatorial phenomenon, in the exhaustive treatise of A. Kechris, V. Pestov and S. Todorcevic \cite{KPT}. Concentration of measure
and Milman-Gromov approach are instead, in a sense, a Ramsey property counterpart for groups that are smooth manifolds \cite{GrMi}.

The Gromov and Milman results allow to single out a fixed point starting from any part of the manifold where the group acts.
Compactness guarantees that fixed points exist but does not localize them and does not show how they become fixed along the incremental dimension process.

We have decided to give a notion (Definition \ref{def2}) of concentration which does not require uniformity as that one proposed in \cite{GrMi}.
Namely, we have emphasized the phenomenon of concentration of a specific sequence of sets, which contains a concentration locus, instead of saying that all sequences of sets concentrate.
However, we show that, under suitable conditions, our definition implies the uniform concentration.

In section \ref{Examples}, two examples show how the measure concentrates from a pure geometrical point of view.
In particular, we single out that the direction of concentration is essential in the non compact case (see Proposition \ref{prop:noncompact}). Remarkably, we find two different directions.
Following the first one, the measure concentrates to a circle, while following the second one, it shrinks to a point. Therefore, even in non compact cases, we may have fixed points but we do not have any uniformity. Indeed, following a wrong direction, we may
not find any fixed point.
In this last case, we have a concentration locus, where the orbits are confined, without fixed point property.
This puts the accent on the fact that extremely amenability and concentration are different phenomena.
Moreover, we show, under an asymmetric action, how the concentration phenomenon could appear.

In all cases it is clear that the concentration on the compact manifold depends on how the concentration locus of the group is mapped to the compact set
under the action of the group.

For our considerations, the relevant extreme amenable groups are only those ones obtained as limits of Levy families, \cite{DoTh}, \cite{GiPe}.

With regard to the connection between the concentration of measure and extreme amenability, the starting point is Theorem 7.1 in \cite{GrMi} which proves extreme amenability of the infinite unitary groups (observe that originally this proof requires an equicontinuous action, later this additional hypothesis has been shown to be unnecessary by Eli Glasner).

It is known that, if one works with the Hilbert-Schmidt topology and considers canonical embeddings, then $U(\infty)=SU(\infty)$ is extremely amenable. Nevertheless, if one considers finer topologies, as the  inductive limit topology
(notice that it is not at all obvious that the limit of a sequence of Lie groups produces Lie groups, see \cite{Gl} and \cite{Ne}), then
$SU(\infty)\subsetneq U(\infty)$, $U(\infty)/SU(\infty)\equiv S^1$ and, obviously, the action of $U(\infty)$ on such $U(1)$ cannot have fixed points, this implies that $U(\infty)$ cannot be extremely amenable (with the finer topology).

By the cited Theorem 7.1 with the inductive limit topology $(U(\infty),U(\infty))$ cannot be Levy since $U(\infty)$ is not extremely amenable.
Observe that, strictly following Gromov-Milman approach, the notion of Levy family does not apply to $U(\infty)$ since it is not a metric space.
However, direct limits of (paracompact) finite-dimensional manifolds with closed embeddings between
them are paracompact, hence normal and, therefore completely regular (see \cite{Lu}). Indeed, arbitrary direct limits of paracompact finite-dimensional smooth manifolds with smooth injective immersions are smoothly paracompact (have smooth partitions
of unity subordinate to any open cover) and hence normal, (see \cite{Gl}), which in turn implies that they are uniform spaces.
Pestov generalizes the notion of Levy family to uniform spaces \cite{Pe1}, so our example makes perfectly sense.
So, $(U(\infty),U(\infty))$ is an example of a group which has a concentration locus ($S^1$) which is not a point. Again, this stresses the difference between the fixed point property and measure concentration.

\smallskip

The paper is organized as follows.
\begin{itemize}
\item In section 1 we outline basic definitions and give first results.
\item In section 2 we show  examples of infinite group action on compact and non compact infinite dimensional manifold. In case of compact one, we prove the uniqueness of fixed point and the exact position.
\item In section 3 we make two physical applications of the above results.
\end{itemize}


\section{Background and first results}

In the first instance, we give the definition of Levy family as it was given by Gromov and Milman in \cite{GrMi}.
\begin{tdef}
For a set $A$ in a metric space $X$ we denote by $N^{\varepsilon}(A)$, $\varepsilon >0$, its e-neighborhood. Consider a family $(X_n,\mu_n)$ with $n:1,2, \dots$ of metric spaces $X_n$ with normalized borel measures $\mu_n$. We call such a family Levy if for any sequence of Borel sets $A_n\subset X_n$  $n:1,2, \dots$, such that $lim\ inf_{n\rightarrow \infty}\ \mu_n(A_n)>0$, and for every $\varepsilon >0$, we have
$lim_{n\rightarrow \infty}\ \mu_n(N^{\varepsilon}(A))=1$.
\end{tdef}
Following the considerations in the introduction we will adopt the following definitions:
\begin{tdef}\label{def2}
 Let $\{X_n, \mu_n\}_{n\in \mathbb N}$ be a family of connected metric spaces with metrics $g_n$, and $\mu_n$ measures w.r.t. which open set are measurable of non-vanishing measure. Assume the measures to be normalized, $\mu_n(X_n)=1$.
 Let $\{S_n\}_{n\in \mathbb N}$ be a family of proper closed subsets, $S_n\subset X_n$. Fix a sequence $\{\varepsilon_n\}_{n\in \mathbb N}$ such that $\varepsilon_n>0$,  $\lim_{n\to\infty}\varepsilon_n=0$, and let
 $\{U^{\varepsilon_n}_n\}_{n\in \mathbb N}$ be the sequence of tubular neighbourhoods of $S_n$ of radius $\varepsilon_n$. We say that the measure concentrates on the family
 $\{S_n\}$ at least at rate of $\varepsilon_n$ if
 \begin{align}
 \lim_{n\to\infty} \mu_n(X_n-U^{\varepsilon_n}_n)=0.
 \end{align}
 We will shortly say that the measure concentrates on $S_n$ and will call it metric concentration. In particular, when $X_n$ are manifolds, we call $S_n$ a concentration locus if it is contained in a submanifold of strictly positive codimension for
 any $n$.
 Moreover, if such a sequence $\varepsilon_n$ converges to 0 at rate $k$ (so that $\lim_{n\to\infty} n^k\varepsilon_n=c$ for some constant $c$), we say that the
 measure concentrates on the family $\{S_n\}$ at least at rate $k$.
\end{tdef}

Notice that in general we may have $\mu_n(S_n)=0$.
Moreover, with this definitions we do not need any notion of convergence of $S_n$ to a final subset. $S_n$ just gives a ``direction of concentration''. Also, our definitions do not pretend to provide any optimality in concentration: it can happen that for
a given sequence of $S_n$ it exists a sequence of proper subsets $S'_n\subset S_n$ on which we still have concentration.
A special example of Definition \ref{def2} consists in the case $X_n=X$ and $S_n=S$ where $X$ is a metric space, $\{\mu_n\}_{n\in\mathbb N}$ a sequence of normalized measures on $X$, compatible with the metric of $X$ and $S\subset X$ a proper closed subset of $X$.

\begin{rem}
  It is apparent that the above definition has some similarities with the usual definition of concentration.
  However, they are different, since our definition is centered on a specific sequence of sets, the concentration locus, while the standard one on generic open sets of positive measure.
  Even the usual notion of concentration has a rate of concentration, represented by the concentration function (see for example \cite{ledoux} for definition and an extensive discussion).
  Again, this notion is based on generic open measure of positive sets instead of a specific sequence of closed sets.
  Roughly speaking the standard definition check if a sequence concentrates and, the concentration function, how it concentrates, our definition says the same thing but specifies where it concentrates.

An interesting rate of concentration is determined by all sequences $\varepsilon_n$ which run to 0 with order less than $1/\sqrt{n}$
(the corresponding concentration function is called ``normal concentration'' \cite{ledoux} ).

\end{rem}

\begin{tdef}{(Gromov-Milman \cite{GrMi})}\\
  A space with a metric $g$ and a measure $\mu$, or an mm-space, is a triple $(X,\mu,g)$, consisting of a set $X$, a metric $g$ on $X$ and a probability Borel measure on the metric space $(X,g)$.
\end{tdef}

We want now to provide a point of contact between the notion of concentration locus and the notion of a sequence of mm-spaces $(X_n)_{n\in N}$ that concentrates to an mm-space $X$.
As far as we know, the latter notion has been originally described in the Gromov's Green Book (\cite{Gro99}, Chapter 3$\frac{1}{2}$H).
Actually, for the definition of observable distance between two mm-spaces $d_{conc}(X,Y)$ and for all notations and definitions needed for the definition of $d_{conc}$, we refer to a more concise presentation in \cite{Shi} or \cite{Sc}.

A sequence of mm-spaces $(X_n)_{n\in N}$ is said to concentrate to an mm-space $X$ if

$$\lim_n d_{conc}(X_n,X)=0$$

Our claim consists in proving, roughly speaking, that whenever a concentration locus has a metric space as limit, the concentration locus concentrates in Gromov sense.
Let $(X,\mu,g)$ and $(X_n,\mu_n,g_n)$ mm-spaces. $(X_n,g_n)$ converges metrically as sequence of metric spaces to $(X,g)$ and
$W_2(\mu_n,\mu)$ converges to 0, in this case we say that $(X_n,\mu_n,g_n)$ converges to $(X,\mu,g)$..
More in detail, assume $X_i\subseteq X_j$ for $i<j$ consider as embedding the natural ones. Let $X=\overline{\bigcup X_i}$ the completion of $\bigcup X_i$ and suppose $X$ is metrizable with metric $d(x,y)=\lim \ d(x_i,y_i)$ with
$\{x_i\}$  and $\{y_j\}$ converging pointwise to $x$ and $y$, respectively. This is exactly what we mean by ``$(X_n,g_n)$ converges metrically as sequence of metric spaces to $(X,g)$".

An example for this $X_i$, token from \cite{Pe1}, is given by the sequence of unitary groups $U_i$, with the natural embeddings. The limit metric is determined by the Hilbert-Schmidt norm $\|\cdot \|_2$, defined in a coordinate-free fashion by
$ \|A\|_2 =\sqrt{{\rm tr}(A^*A)}$ for $ A\in Mat(n\times n)$.
This closes to the set of compact operators $K$ such that the eigenvalues of $\sqrt{K^* K}$ are the coordinates.
If this sequence is in $\ell^2$, then $K$ is said to be of Schatten 2-class. The operators of Schatten 2-class are exactly those
for which the Hilbert-Schmidt norm, as given by the above equation, makes sense.

We denote by $U(\infty)_2$ the collection of all unitary operators of the form $I+K$, where $I$ is the identity and $K$ as above.
The Hilbert-Schmidt metric on the group $U(\infty)_2$ , given by the rule $d_{HS}(u,v) =\|u-v\|_2$, is (well-defined and) bi-invariant.
With this metric, $U(\infty)_2$ is a Polish group.
For every $n\in N$, the unitary group $U(n)$ of rank $n$ belongs into $U(\infty)_2$ in the usual way:
by enlarging the matrix with 1 along the diagonal. The union $\bigcup U(n)$ is dense in $U(\infty)_2$.
$U(\infty)_2$ is the completion of the abstract group $\bigcup U(n)$  with regard to the Hilbert-Schmidt metric.

Let $(S_n,\nu_n)$ be a concentration locus for $(X_n,g_n,\mu_n)$, with a defined measure $\nu_n$, for such a sequence. Let $(S,g)$ a subspace of $(X,g)$ and $(S,\nu,g)$ an mm-space such that $(S_n,\nu_n,g_n\downharpoonright_{S_n})$ converges to $(S,\nu,g)$.

The notion and the definition of optimal transport can be found in many text books, for a quick exposure, in \cite{Sc} or in \cite{Shi}, for an extensive treatise you can refer to \cite{VilBook}.

In the following we shall use optimal transports $p$ from $(X,d,\mu)$ to $S$ such that $d(x,p(x))=d(x,S)$, we denote such transports as "transports to the border".

\begin{prop}\label{ConLocConcSpace}

  Let $(X_n,\mu_n,g_n)$ a sequence of manifolds with bounded curvature which converges to $(X,\mu,g)$ , $(S_n,\nu_n)$ a concentration locus for $(X_n,\mu_n,g_n)$ and, for each $n$, $p'_n$ be a continuous optimal transport to the border from $\mu_n$
  to $\nu_n$ such that $W_2(\nu_n,\mu_n)$ converges to 0.
  If $(S_n,\nu_n,g_n\downharpoonright_{S_n})$ converges to a manifold $(S,\nu,g)$ with bounded curvature, then $(X_n,\mu_n,g_n)$ concentrates to $(S,g,\nu)$ in Gromov sense.
\end{prop}
\begin{proof}
This proof does not make any use of Levy family either condition of first eigenvalue \cite{GrMi}.

We use \cite{Shi} characterization for concentration Corollary 5.36 or Theorem 2.2 \cite{Sc} .

Let $p'_n$ the continuous (in particular Borel) optimal transport of $X_n$ on $S_n$, $i_n$ the isometric injection of $S_n$ in $S$ and $p_n=i_n\circ p'_n$. By our hypothesis $\sharp p'_n\mu_n=\nu_n$.

$W_2(\mu_n,\nu)\le W_2(\mu_n,\nu_n)+W_2(\nu_n,\nu)\le$ (by hypothesis for large enough $n$) $\le \epsilon$.

By Lemma 9.12 \cite{Shi} $\mu_n$ weakly converges to $\nu$ and this verifies condition 1 of Corollary 5.36 of Shioya.

Regarding condition 2 of Corollary 5.36 of Shioya\cite{Shi}, it is sufficient to prove that for a given $f\in Lip_1(S,g)$ there exists for a sufficient large $n$ an $f'_n\in Lip_1(X_n)$ such that $me_{\mu_n}(f'_n,f\circ p_n))<\epsilon$. On the other side for a
given $f'_n\in Lip_1(X_n,g)$ we have to prove that there exists an $f\in Lip_1(S,g)$ such that $me_{\mu_n}(f'_n,f\circ p_n))<\epsilon$.

First we consider $\tilde{f'}_n=f\mid_{S_n}$ which is in $Lip_1(S_n)$.
By \cite{Kirsz} we can extend $\tilde{f'}_n$ to an $f'_n\in Lip_1(X_n)$. Let us consider open sets $U^{\epsilon_n}_n$ around $S^n$.
Choose an $n_1$ such that for all $n>n_1$ $\mu_n(X_n\setminus U^{\epsilon_n}_n)<\epsilon$. Observe that $f\circ p_n$ inside $S^n$ coincides with $f'_n$ since optimal transport doesn't move points inside the destination support.
Therefore we have to check the measure of $\{x\mbox { such that } \mid f'_n(x)-f\circ p_n\mid\ge\epsilon\}$ just outside of $S_n$. In $X_n\setminus U^{\epsilon_n}_n$ the measure is less than $\epsilon$.
While inside $U^{\epsilon_n}_n\setminus S_n$ for each point $x$ $g_n(x,p_n'(x))=g_n(x,S_n)$ which is less than $\epsilon_n$.
We can choose $n_2>n_1$ in such a way $\epsilon_n<\epsilon$, since $f'_n\in Lip_1(X_n)$ we get $\mid f'_n(x)-f\circ p_n (x)\mid <\epsilon$.

On the other side, consider a restriction $\tilde f'_n$ of $f'_n\in Lip_1(X_n,g)$ to $S_n$ then again by \cite{Kirsz} extend $\tilde f'_n$ by an $f\in Lip_1(S)$. Using a similar argument as above, we get $\mid f'_n(x)-f\circ p_n (x)\mid <\epsilon$ inside $U^{\epsilon_n}_n\setminus S_n$, $\mu_n(X_n\setminus U^{\epsilon_n}_n)<\epsilon$ and $x\in S_n$ left fixed by the optimal transport.

\end{proof}

\begin{rem}
  In an on-going paper we proved that a large class of Compact Lie Groups\footnote{For example, $SU(n)$ belongs to this class} satisfies the conditions for applying Proposition \ref{ConLocConcSpace}. In case of Hilbert Schmidt distance, $S$ is a point.
  Indeed, Schneider proved the following Pestov's conjecture [\cite{Pe1} , Conjecture 7.4.26]:

If $G$ is a metrizable topological group, equipped with a compatible right-invariant metric $d$,
and $(K_n)_{n\in N}$ is an increasing sequence of compact subgroups such that
  \begin{itemize}
    \item the union $\bigcup_{n\in N}K_n$ is everywhere dense in $G$, and
    \item $(K_n, d\mid K_n, \mu_n)_{n\in N}$ concentrates to a fully supported, compact mm-space $(X, d_X, \mu_X)$,
where $\mu_n$ denotes the normalized Haar measure on $(K_n)_{n\in N}$,

   \end{itemize}
then the topological space $X$ supports the structure of a $G$-flow, with respect to which it
admits a morphism to every $G$-flow.

\end{rem}

As a consequence, whenever $X_n$ are compact and the limit of a concentration locus is minimal, by Schneider theorem, the limit space $S$ is an Universal Minimal flow.
In case it should be a point the sequence is a Levy family hence $X$ is extremely amenable (\cite{Shi},\cite{Sc}). Moreover, under conditions of Theorem \ref{ConLocConcSpace}, if a concentration locus concentrates to a point then $X_n$ is a Levy family.


\section{Concentration of measure on infinite dimensional compact sets}\label{Examples}
The extreme amenability of a group is related to the concentration of the measure induced by the Levy family on any compact Lie group.
In order to improve our intuition on the phenomenon, we  will consider two simple examples of infinite dimensional compact sets and the concentration of measure on them associated to the continuous action of certain infinite dimensional groups.


\subsubsection{An elementary example}
Let $\mathcal H$ be a real infinite dimensional Hilbert space, with scalar product $(|)$ and norm $\|\ \|$. We define the spheres and the closed balls
\begin{align}
& S^\infty_r:=\{x\in \mathcal H | \|x\|=1 \}, \\
& B^\infty_r:=\{x\in \mathcal H | \|x\|\leq1 \}.
\end{align}
Both are closed and bounded but none of them is compact in $\mathbb H$ with the strong topology. Therefore, let us consider the weak topology, the weakest topology which makes the maps of the topological dual $\mathcal H'$ of $\mathcal H$
continuous. Its open set are generated by open sets of the form
\begin{align}
 U_{v,\varepsilon}(x_0)=\{x\in \mathcal H | |(v|x-x_0)|<\varepsilon \},
\end{align}
under finite intersections and arbitrary unions, where $v\in \mathcal H$.\\
If $\mathcal H_v\equiv v^\perp$, $v\neq 0$, then
\begin{align}
 U_{v,\varepsilon}(x_0)=\bigcup_{\alpha\in \mathbb R, {|\alpha|<\varepsilon}} \{x_0+\alpha \frac v{\|v\|^2} +\mathcal H_v\}.
\end{align}
Of course, $S^\infty_r$ is not closed in the weak topology, and $\overline{S^\infty_r}=B^\infty_r$. Indeed, let us first consider $y\in B^\infty_r$, and fix an orthonormal complete system $\{e_j\}_{j=1}^\infty$ for $\mathcal H$
such that $y=\|y\| e_1$. Therefore, the sequence
\begin{align}
 & x_n:=y+\sqrt {r^2-|y|^2} e_n, \quad n>1
\end{align}
is in $S^\infty_r$ and satisfies
\begin{align}
 \lim_{n\to\infty} (v|y-x_n)=0
\end{align}
for any $v\in \mathcal H$, which means that $y$ is an accumulation point for $S^\infty_r$ in the weak topology. On the opposite, if $y\notin B^\infty_r$, then $\|y\|>r$ and, for any $x\in S^\infty_r$,
\begin{align}
 |(y|y-x)|=|\|y\|^2 -(y|x)|\geq |\|y\|^2 -|(y|x)||\geq \|y\|(\|y\|-\|x\|)=\|y\|(\|y\|-r)>0,
\end{align}
so that $y$ is isolated from the ball.

\

\noindent Let us now fix a complete orthonormal system $\{e_j\}_{j=1}^\infty$ for $\mathcal H$. Set
\begin{align}
 \mathbb R^N:=\mathbb R e_1\oplus \cdots \oplus \mathbb R e_N.
\end{align}
Notice that its orthogonal complement in $\mathbb H$ is
\begin{align}
 {\mathbb R^N}^\perp=\mathcal H_{e_1} \cap \cdots \cap \mathcal H_{e_N}.
\end{align}
We can consider the linear action of the group $SO(N)$ on $\mathcal H$ defined by its fundamental action on $\mathbb R^N$ and by its trivial action on ${\mathbb R^N}^\perp$. Let us consider the sequence of canonical embeddings
\begin{align}
 \ldots \subset SO(N) \subset SO(N+1) \subset SO(N+2) \subset \ldots
\end{align}
and its limit $SO(\infty)$ w.r.t. to any topology which makes the embeddings $SO(N)\hookrightarrow SO(\infty)$ continuous for all $N$. Let be $K=B^\infty_r$ endowed with the weak topology, so that $K$ is a compact set. For any given
point $x\in B^\infty_r$ we define the measures $\mu^x_N$ over $K$ induced by the action of $SO(N)$ over $K$ as follows. If
\begin{align}
 f: SO(N)\times \mathcal H \longrightarrow \mathcal H
\end{align}
is the action of $SO(N)$ on $\mathcal H$ defined above,
it is clear that its restriction over $SO(N)\times B^\infty_r$ defines an action of $SO(N)$ on $B^\infty_r$. Given a subset $A\subseteq B^\infty_r$, let us consider the subset
$SO(N)^x_A$ of $SO(N)$ defined by
\begin{align}
 SO(N)^x_A:=\{ g\in SO(N) | f(g,x)\in A \}.
\end{align}
We say that $A$ is measurable if $SO(N)^x_A$ is measurable with respect to the normalized Haar measure $\mu_{SO(N)}$ of $SO(N)$,\footnote{The normalization is $\mu_{SO(N)}=1$.} and put
\begin{align}
 \mu^x_N(A):=\mu_{SO(N)} (SO(N)^x_A).
\end{align}

\begin{rem}
  In \cite{CU22} we calculate explicitly a concentration locus for classical compact Lie groups, in particular for $SO(n)$.
\end{rem}

\begin{teo}
The following holds:
\begin{align}
 \lim_{N\to\infty} \mu^{e_1}_N=D_0
\end{align}
where $D_0$ is the Dirac measure with support in $0$.
\end{teo}

\begin{proof}
The image of $e_1$ under the action of $SO(N+1)$ is
\begin{align}
f(SO(N+1),e_1)=S^N_{1}=S^\infty_{1}\cap \mathbb R^N=\{ y=\sum_{j=1}^{N+1} y_j e_j | \|y\|=1 \}.
\end{align}
It follows immediately that
\begin{align}
\mu^{e_1}_{N+1}=D_{S_1^N}
\end{align}
the Dirac measure uniformly supported on $S_1^N$ (so that it is the normalized Lebesgue measure when restricted to the support). Choosing for any $N$ the spherical polar coordinates with azimuthal axes defined by $e_1$, it is a standard argument
(see \cite{Le}) to show that when $N\to\infty$ the measure concentrates on $S^\infty_1 \cap \mathcal H_{e_1}$. This is exactly the image of a concentration locus of the $SO(N+1)$ family: after identifying $SO(N+1)$ with the $SO(N)$ Hopf fibration
over $S^N$ and $SO(N)$ with the isotropy group of $e_1$ under the action of $SO(N)$, following the results in \cite{CU22}, it is easy to see that a concentration locus on the group can be identified with the restriction of the Hopf fibration to an equator of $SO(N+1)/SO(N)\simeq S^N$.
This family of fibrations is mapped to the family of the equators of the orbits of $SO(N+1)$ through $e_1$. \\
Iterating the procedure w.r.t. $e_2\in S^\infty_1 \cap \mathcal H_{e_1}$ and so forth, we see that for any finite $k$ the measure
concentrates on $S^\infty_1 \cap {\mathbb R^k}^\perp$. Taking the closure of these sets and noting that $\bigcap_{k\in \mathbb N} \mathcal H_k=0\in B^\infty_1$. We get the assertion.
\end{proof}
\begin{rem}\rm
 The proof not only gives the assertion, but also shows that the concentration locus has growing codimension in the family of $SO(N+1)$, and when we take the limit its image on the spheres collapses to the fixed point $0$. So, while the concentration
locus has no limit on the groups, it has a well defined limit on the compact spheres on which the groups act.
\end{rem}
This phenomenon of concentration of measure on a compact space can be also interpreted as an optimal transport of mass along geodesics of measures, which converge to $D_0$, see, for example, \cite{LoVi}. Such flow of measures can be geometrically
visualized in our case as a geodesic flow transporting the equators toward the nord pole $e_1$ along meridian geodesics. In particular, a recent work by Schneider
\cite{Sc} uses this approach to prove a conjecture by Pestov \cite{Pe1}.
\begin{prop}
If $x\in B^\infty_r$, then
\begin{align}
 \lim_{N\to\infty} \mu^{x}_N=D_0.
\end{align}
\end{prop}
\begin{proof}
 If $x\neq 0$ then the proof follows exactly the same line as above. The case $x=0$ follows from $\mu^{0}_N=D_0$ for any $N$.
\end{proof}

Notice that we can get a completely geometrical picture of this phenomenon: acting on $x\in K=B^\infty_r$, $SO(\infty)$ generates $K$ as a union of (infinitely many) layers like an onion when $\|x\|$ varies in $[0,r]$. In fact, $SO(\infty)$ acts transitively on
each layer, but all layers have a common center, the origin $0$, which is left fixed by each $SO(N)$. Since the normalised measure induced on each $S^N_{|x|}$ is always the same as for $S^N_1$ if $x\neq 0$, it is natural to expect that the measure
will concentrate on the common accumulation point, which is just $0$.

\

The same result can be extended moving to a complex Hilbert space $\mathcal H^{\mathbb C}$ and acting with the groups $SU(N)$, where projective spaces replace the spheres.
More precisely, the reference compact for $SU(N+1)$ is $\mathbb {CP}^N$. Again, for each $N$ and $x\in K$ (the weak closure of the $\mathbb {CP}^\infty$) we can define the measure $\mu^x_N$ and get
\begin{prop}\label{prop:fixed}
If $x\in K$, then
\begin{align}
 \lim_{N\to\infty} \mu^{x}_N=D_0.
\end{align}
\end{prop}
We do not know whether $SO(\infty)$ and $SU(\infty)$ are extremely amenable or not, if endowed with the inductive topology. Nevertheless, we know that $U(\infty)$ is not extremely amenable with such topology. Despite this, probably as a consequence
of the weakness of the weak topology, its action on $K$ results in the concentration of the measure in a point. This concentration is ``non essential'' in the sense that it is due to the property of the compact and not of the Levy sequence of groups.
This is easily seen by changing the compact set. If $K=B_r^\infty$ with the weak topology and we identify $S^1$ endowed with the metric topology with $U(\infty)/SU(\infty)$, then $K'=K\times S^1$ with the product topology is a compact set. Now, after fixing
a point $x'=(x,\theta)\in K'$, we see that, because of Proposition \ref{prop:fixed}, necessarily the measures generated by the sequence $SU(N)'s$ concentrate on the point $x'_0=(0,\theta)$. But since
$U(\infty)$ has no fixed points on $S^1=U(\infty)/SU(\infty)$, then the measures induced by the sequence of $U(N)'s$ will obviously concentrate uniformly on ${0}\times S^1$.

We can now state a simple but interesting proposition, giving us some information when working on a non compact infinite dimensional set.
\begin{prop}\label{prop:noncompact}
 Let $\{K_n\}_{n\in \mathbb N}$ be a sequence of sets $K_n\subset K_{n+1}$, over which acts the family of groups $\{G_n\}$, endowed with an invariant measure. Let $G_\infty$ the limit group defined by the sequence $G_n\subset G_{n+1}$.
 Assume that there exists a concentration locus $c_n\subset K_n$, such that cutting $K_n$ along $c_n$, it is divided into disconnected parts $K^+_n$ and $K^-_n$, having the following properties:
 \begin{itemize}
 \item $c_n\subset K^+_n$ and $K_n=K^+_n\cup K^-_n$;
 \item for $n$ large enough all $K^+_n$ are contained in a compact subset $K^+$ of $K$;
 \item it exists a sequence of subgroups $F_n\subseteq G_n$ acting on $K^+_n$ and having the same limit $G_\infty$.
 \end{itemize}
 Then, if $G_\infty$ is extremely amenable, $K\equiv \bigcup_n K_n$ admits a fixed point.
\end{prop}
\begin{proof}
 Since $K^+$ is compact and $F_n$ act on $K^+$, we get that $G_\infty$ has a fixed point in $K^+$, and therefore in $K$.
\end{proof}
This shows how eventual fixed points could be individuated even on a non compact set $K$. It is clear that to do it, it's essential to have a well posed definition of concentration locus.


\subsection{Linear non symmetric action}
Let $a\in \mathbb R^{N+2}$, and let $\bar B_{N+2}$ be the closed ball
\begin{align}
\bar B_{N+2}=\{x\in \mathbb R^{N+2} | \|x\|\leq 1 \},
\end{align}
while
\begin{align}
B_{N+2}=\{x\in \mathbb R^{N+2} | \|x\|<1 \},
\end{align}
is the set of its internal points. We call $\alpha=\|a\|$ and assume
\begin{align}
 0<\alpha <1.
\end{align}
Finally, we call
\begin{align}
 S_{N+1}=\partial \bar B_{N+2}
\end{align}
that is a $(N+1)$-sphere of radius $1$, centred in $o\equiv 0$. We now define a linear action of $SO(N+2)$ on $S_{N+1}$, which we simply call {\it the non symmetric action}. This is a left action
\begin{align}
 \nu_N: SO(N+2)\times S_{N+1} \longrightarrow S_{N+1}
\end{align}
defined as follows. Let $p\in S_{N+1}$, and consider the right half line $r$ starting from the origin $a$ through the point $p$. If $R\in SO(N+2)$, the right half line $R(r)$ will meet $S_{N+1}$ in a single point $q$. We define
\begin{align}
 \nu_N(R,p):=q.
\end{align}
Since $0<\alpha<1$, it is well defined and gives rise to a transitive action of $SO(N+2)$ on $S_{N+1}$. \\
Let us now determine the invariant measure on the sphere induced by the action $\nu_N$. First notice that, like for the symmetric action, the isotropic subgroup of a point $p\in S_{N+1}$ is the $SO(N+1)$ fixing the segment $op$.
The complementary generators originate the displacements defining the measure. We can easily compare the resulting measure $d\mu_a$ with the standard measure $d\mu_0$.
Let $\phi$ be the angle between the segments $oa$ and $op$ with orientation from $oa$ to $op$. Then
\begin{prop}
With the above notations we have
\begin{align}
d\mu_a=\frac {1+\alpha\cos\phi}{(1+\alpha^2+2\alpha\cos \phi)^{\frac N2+1}} d\mu_0,
\end{align}
where $\mu_0$ is the Lebesgue measure over $S_{N+1}$ normalised so that $\mu_0(S_{N+1})=1$. 

In particular $\mu_a(S_{N+1})=1$.
\end{prop}
\begin{proof}
The normalised measure $d\mu_0$ is nothing but the element of solid angle as seen by $o$, divided by the total solid angle. By construction, $d\mu_a$ is the element of solid angle as seen by $a$, again divided by the total solid angle, so that we have
of course $\mu_a(S_{N+1})=1$. Now fix a point p on $S_{N+1}$. By definition the relation between the two measures in $p$ is
\begin{align}
d\mu_a(p) =\frac {\cos \psi}{R^{N+1}} d\mu_0,
\end{align}
where $R=\|p-a\|$ and $\psi$ is the angle $o\hat pa$. The cosines theorem says us that
\begin{align}
 R^2=1+\alpha^2 +2\alpha \cos \phi.
\end{align}
As for $\cos \psi$, we can again apply the cosines theorem to the triangle $apo$ in the form
\begin{align}
\|oa\|^2= \|op\|^2+\|pa\|^2 -2\|pa\| \|op\| \cos \psi.
\end{align}
Using that $\|oa\|=\alpha$, $\|op\|=1$, and $\|pa\|=R$, we get
\begin{align}
 \cos \psi= \frac {1+\alpha\cos\phi}{(1+\alpha^2+2\alpha\cos \phi)^{\frac 12}},
\end{align}
and the proposition is proved.
\end{proof}


\subsection{The extremal infinite sphere and the concentration of measure}\label{5.4}
Let $\mathcal H$ be a separable real Hilbert space, and $\{e_n\}_{n=0}^\infty$ be a complete orthonormal system (CONS) for it. Let $a\in \mathcal H$, with $\alpha:=\|a\|=1$, we define the {\it extremal sphere}
\begin{align}
 S_a:=\{x\in \mathcal H | \|x\|=1 \}.
\end{align}
Here the $a$ means only that the sphere has the point $a$ as a marked point.
We assume to choose the CONS in such the way that
\begin{align}
 a_n=(e_n|a) \neq 0 \qquad \forall n.
\end{align}
For any given $N=0,1,2,\ldots$, we identify
\begin{align}
\mathbb R^{N+2}&\equiv \mathcal H_N=\{ x\in \mathcal H | (e_n| x)=0,\ \forall n>N+1 \}, \\
S_{\underline a_N} &=\mathcal H_N\cap S_a, \quad \underline a_N=(a_0, \ldots, a_{N+1}).
\end{align}
By hypothesis, $0<\|\underline a_N\|\equiv \alpha_N<1$, so that $\underline a_N$ is internal to $S_{{\underline a}_N}$. Let $O(\mathcal H)$ be the unitary group of $H$. The subgroup fixing $\mathcal H_N^\perp$ is $O(N+2)$, the orthogonal group acting
on $\mathbb R^{N+2}$. $SO(N+2)$ is the subgroup fixing the orientation of $(e_0,\ldots, e_{N+1})$. Let us consider the sequence
\begin{align}
 SO(\infty):=(SO(2)\subset SO(3) \subset \ldots\subset SO(N) \subset SO(N+1) \subset \ldots)
\end{align}
defined by the identifications above. We see that on each $S_{\underline a_N}$ we can define a left action $\nu_{\underline a_N}$ of $G_N:=SO(N+2)$ on $S_{\underline a_N}$, defined as in previous section. Notice that the action of $G_N$ over
$S_{\underline a_N}$ is just the restriction to
$S_{\underline a_N}$ of the action on $S_{\underline a_{N+1}}$ of the subgroup of $G_{N+1}$ leaving fixed the segment $a_{N+1}a_N$.\\
The measure induced on $S_{\underline a_N}$ by $G_N$ is
\begin{align}
d\mu_{\underline a_N}(p)=\frac {1+\alpha_N\cos\phi_N}{(1+\alpha_N^2+2\alpha_N\cos \phi_N)^{\frac N2+1}} d\mu_{N,0}(p),
\end{align}
where $\phi_N$ is the azimuthal angle of $p$ w.r.t. $ao$, and $d\mu_{N,0}(p)$ is the Lebesgue measure on $S_{a_N}$ normalised to 1. We get:
\begin{prop}
With the above notation we have
\begin{align}
\lim_{N\to\infty} \mu_{\underline a_N}=D_a,
\end{align}
where $D_a$ is the Dirac measure with support in $a$. The action of $SO(\infty)$ on $S_a$ has fixed point in $a$.
\end{prop}
\begin{proof}
Using standard polar coordinates on the sphere, w.r.t. the azimuthal angle $\phi_N$, we can write
\begin{align}
d\mu_{\underline a_N}(p)=\frac {1+\alpha_N\cos\phi_N}{(1+\alpha_N^2+2\alpha_N\cos \phi_N)^{\frac N2+1}} \sin^{N} \phi_N d\phi_N d\mu_{N-1,0}(p).
\end{align}
On the other hand, if $\theta=a\hat p o$ and $R(\phi_N)=\|ap\|=\sqrt {1+\alpha^2+2\alpha \cos \phi_N}$, we have
\begin{align}
 \sin \phi_N=R(\phi_N) \sin \theta.
\end{align}
From this we get
\begin{align}
 \frac {1+\alpha_N\cos\phi_N}{(1+\alpha_N^2+2\alpha_N\cos \phi_N)^{\frac N2+1}} \sin^{N} \phi_N d\phi_N=\sin^N \theta \ d\theta.
\end{align}
Therefore, in this coordinates
\begin{align}
d\mu_{\underline a_N}(p)= \sin^N \theta \ d\theta\ d\mu_{N-1,0}(p).
\end{align}
This is the standard spherical measure, which is well known to concentrate on the equatorial sphere $\theta=\frac \pi2$. On $S_{N+1}$ this is a $S^N$ sphere with center in $\underline a_N$ and radius $\sqrt {1-\|\underline a_N\|^2}$.
Since $\underline a_N \to a$, and $\|a\|=1$, we get the assert.
\end{proof}

\begin{rem} \rm
$S_a$ is not a compact space in the topology of $\mathcal H$.
Notice that this result is nothing but a consequence of Proposition \ref{prop:noncompact}: here the equators play the role of $c_n$ of the proposition. Again, the equators are just the image of the concentration locus of $SU(N+1)$ mapped on the sphere
now through the non symmetric action. Under the action the locus is shrunk down to $\underline a$.
\end{rem}


\subsection{Pictorial idea}\label{pictorialx}

We want now to give a pictorial vision to the construction done above. This is because it will help us to better understand what happens and to proceed further to the compact case. The idea of the asymmetric action of rotation is depicted in figures
\ref{caso D1-1} and \ref{caso D1-2}. They show the relation between the rotation angle $\theta$ and the azimuthal angle $\phi$ of spherical coordinates. When the segment $ap$ belongs to the plane of the rotation (assuming a plane rotation), the main relation
is $\sin \phi=\|ap\|\sin \theta$.

\begin{figure}[!htbp]
\begin{center}
\begin{tikzpicture}[rounded corners]
\draw [red,dashed] (-1,-2.5) -- (1.48556,1.48556*2.5);
\draw [thick] (0.5,1.25) arc (68.2:0:1.346cm);
\draw [thick] (1.5,0) arc (0:-7.18:1.5cm);
\draw [thick] (-0.5,-1.25) arc (68.2:26.565:1.346cm);
\draw [thick] (1.34164-1,0.670819-2.5) arc (26.565:21.926:1.5cm);
\draw (-1,-2.5) -- (4,0);
\draw (0,0) -- (4,0);
\draw (-1,-2.5) -- (3.96863,-0.5);
\draw (0,0) -- (3.96863,-0.5);
\draw [ultra thick] (0,0) circle (4cm);
\draw [fill] (0,0) circle (1.5pt);
\draw [red,fill] (-1,-2.5) circle (1.5pt);
\draw [cyan,fill] (4,0) circle (1.5pt);
\draw [cyan,fill] (3.96863,-0.5) circle (1.5pt);
\node at (-0.2,0) {$\pmb o$};
\node [red] at (-1.2,-2.5) {$\pmb a$};
\node [cyan] at (4.2,0) {$\pmb p$};
\node [cyan] at (4.6,-0.5) {$\pmb {R_{d\theta}(p)}$};
\node at (1.3,0.9) {$\pmb \phi$};
\node at (1.5,-0.4) {$\pmb {d\phi}$};
\node at (0,-1.3) {$\pmb \theta$};
\node at (0.5,-2.2) {$\pmb {d\theta}$};
\end{tikzpicture}
\caption{Relation between the rotation angle $d\theta$ and the polar angle $d\phi$ for the one dimensional sphere. $a$ and $o$ stay on the same plane.}\label{caso D1-1}
\end{center}
\end{figure}

When the rotation is on a plane orthogonal to $ao$, then the relation is just $\|op\| d\psi=\|ap\| d\chi$. Notice that in this case the length of $op$ is preserved.

\begin{figure}[!htbp]
\begin{center}
\begin{tikzpicture}[rounded corners]
\draw [red,dashed] (0,-3.5) -- (0,0);
\draw [cyan] (0,-3.5) -- (3.6,0.45);
\draw [cyan] (0,-1.5) -- (3.6,0.45);
\draw [cyan] (0,-3.5) -- (3.6,-0.45);
\draw [cyan] (0,-1.5) -- (3.6,-0.45);
\draw [ultra thick] (0,0) circle [x radius=4cm, y radius=1cm];
\draw [fill] (0,-1.5) circle (1.5pt);
\draw [red,fill] (0,-3.5) circle (1.5pt);
\draw [cyan,fill] (3.6,0.45) circle (1.5pt);
\draw [cyan,fill] (3.6,-0.45) circle (1.5pt);
\node at (-0.2,-1.5) {$\pmb o$};
\node [red] at (-0.2,-3.5) {$\pmb a$};
\node [cyan] at (3.9,0.5) {$\pmb p$};
\node [cyan] at (4.3,-0.5) {$\pmb {R_{d\chi}(p)}$};
\node at (1.5,-0.4) {$\pmb {d\psi}$};
\node at (0.7,-2.4) {$\pmb {d\chi}$};
\end{tikzpicture}
\caption{Relation between the rotation angle $d\chi$ and the polar angle $d\psi$ when the circle is perpendicular to $oa$: $op\ d\psi=ap\ d\chi$.}\label{caso D1-2}
\end{center}
\end{figure}

Using this relations and decomposing arbitrary rotations in terms of rotation on planes containing $ao$ and rotations in hyperplanes orthogonal to $ao$, one easily computes the measure induced by the action $\nu_a$ on a sphere.
This is shown in figure \ref{rotazioni S2}.

\begin{figure}[!htbp]
\begin{center}
\begin{tikzpicture}[rounded corners]
\draw[ultra thick,black!30!white] (0,0) circle [x radius=4cm, y radius=1cm];
\draw[ultra thick,black!30!white] (0,0) circle [x radius=1cm, y radius=4cm];
\draw[ultra thick,black!30!white] (0,-2) circle [x radius=3.4cm, y radius=0.8cm];
\draw [dashed,red] (0,4) -- (0,-4);
\draw [thick, blue!40!white] (0,0) -- (0.836,2.2);
\draw [thick, blue!40!white] (0,0) -- (0.915,1.6);
\draw [thick, blue] (0,3) -- (0.836,2.2);
\draw [thick, blue] (0,3) -- (0.915,1.6);
\draw [thick, orange!40!white] (0,0) -- (-3.5,-0.485);
\draw [thick, orange!40!white] (0,0) -- (-3,-0.66);
\draw [thick, orange] (0,3) -- (-3.5,-0.485);
\draw [thick, orange] (0,3) -- (-3,-0.66);
\draw [thick, green!40!white] (0,0) -- (-0.4,-2.79);
\draw [thick, green!40!white] (0,0) -- (-1.1,-2.755);
\draw [thick, green] (0,3) -- (-0.4,-2.79);
\draw [thick, green] (0,3) -- (-1.1,-2.755);
\shade[ball color=red] (0,3) circle (1.5pt);
\shade[ball color=black!50!white] (0,0) circle (1.5pt);
\node at (0.2,0) {$\pmb o$};
\node [red] at (0,3.2) {$\pmb a$};
\shade[ball color=blue,opacity=0.5] (0,0) circle (4cm);
\draw [thick] (-4,0) arc (180:360:4cm and 1cm);
\draw [thick] (0,4) arc (90:-90:1cm and 4cm);
\draw [thick] (-3.37,-2.1) arc (187.2:352.8:3.4cm and 0.8cm);
\shade[ball color=blue] (0.836,2.2) circle (1.5pt);
\shade[ball color=blue] (0.915,1.6) circle (1.5pt);
\shade[ball color=orange] (-3.5,-0.485) circle (1.5pt);
\shade[ball color=orange] (-3,-0.66) circle (1.5pt);
\shade[ball color=green] (-0.4,-2.79) circle (1.5pt);
\shade[ball color=green] (-1.1,-2.755) circle (1.5pt);
\node [blue] at (1.1,2.2) {$\pmb p$};
\node [blue] at (1.5,1.6) {$\pmb {R_1(p)}$};
\node [orange] at (-3.6,-0.7) {$\pmb q$};
\node [orange] at (-2.9,-0.99) {$\pmb {R_2(q)}$};
\node [green] at (-0.4,-3.1) {$\pmb r$};
\node [green] at (-1.1,-3) {$\pmb {R_3(r)}$};
\end{tikzpicture}
\caption{Rotations decompose on the ones along planes containing $oa$ (like $R_1$) and the ones on hyperplanes through $o$ perpendicular to $oa$ (like $R_2$ and $R_3$).}\label{rotazioni S2}
\end{center}
\end{figure}

One can then inductively generate a sequence of actions $\nu_{\underline a_N}$ on a sequence of spheres $S_{\underline a_{N+1}}$, $S_{\underline a_{N}}\subset S_{\underline a_{N+1}} \subset S_{\underline a_{N+2}}$, as done in the previous section
and depicted in figure \ref{fig:caso generale}.

\begin{figure}[!htbp]
\begin{center}
\begin{tikzpicture}[rounded corners]
\draw [fill,black!50!white] (0,-3.9) circle [x radius=0.8pt, y radius=0.5pt];
\draw[ultra thick,black!30!white] (0,0) circle [x radius=4cm, y radius=1cm];
\draw[dashed] (0,-3.9) -- (0,3.9);
\draw[dashed] (0,-2) circle [x radius=3.4cm, y radius=0.8cm];
\draw [black!70!white, -latex, thick] (0,0) -- (0.75,0.2);
\node [black!70!white, -latex, thick] at (0.4, 0.3) {$\pmb {e_1}$};
\draw [black!70!white, -latex, thick] (0,0) -- (0.65,-0.3);
\node [black!70!white, -latex, thick] at (0.32,-0.3) {$\pmb {e_0}$};
\draw [black!70!white, -latex, thick] (0,0) -- (0,0.75);
\node [black!70!white, -latex, thick] at (-0.25,0.6) {$\pmb {e_2}$};
\draw[red, dashed] (-1.5,-0.4) -- (-1.5,-2.4);
\draw[red, dashed] (-1.5,-0.4) -- (0,0);
\draw[red, dashed] (0,0) -- (-1.5,-2.4);
\draw [thick] (-0.75,-0.2) arc (40:-0.8:0.4cm);
\draw [thick] (-1,-1.6) arc (100:6:0.2cm);
\draw[green!70!black,thick] (-1.5,-2.4) -- (3.2,1.5);
\draw[blue,thick] (-1.5,-0.4) -- (3,-0.665);
\shade[ball color=red] (-1.5,-0.4) circle (1.5pt);
\shade[ball color=red] (-1.5,-2.4) circle (1.5pt);
\shade[ball color=black!70!white] (0,0) circle (1.5pt);
\node[red] at (-1.8,-0.4) {$\pmb {\underline a_0}$};
\node[red] at (-1.8,-2.4) {$\pmb {\underline a_1}$};
\node at (-0.2,0.1) {$\pmb o$};
\node at (-0.7,0) {$\pmb {\phi_0}$};
\node at (-0.7,-2) {$\pmb {\phi_1}$};
\shade[ball color=blue,opacity=0.5] (0,0) circle (4cm);
\draw [fill] (0,3.9) circle [x radius=0.8pt, y radius=0.5pt];
\draw [ultra thick] (-4,0) arc (180:360:4cm and 1cm);
\draw [fill,blue] (3,-0.665) circle [x radius=0.8pt, y radius=1.3pt];
\draw [fill,green!30!black] (3.2,1.5) circle [x radius=1pt, y radius=1.3pt];
\node[green!30!black] at (3.8,1.7) {$\pmb {p\in S_2}$};
\node[blue] at (3.6,-0.8) {$\pmb {p\in S_1}$};
\end{tikzpicture}
\caption{Asymmetric actions in successive dimensions. The picture represents the 1 and 2 dimensional cases, but one can imagine to consider the general case after replacing $0,1,2$ with $N, N+1, N+2$ everywhere.
The points $\underline a_0\equiv (a_0,a_1)$ and $a_1\equiv (a_0,a_1,a_2)$ are the centres of the rotations action on $S_1$ and $S_2$ respectively.}\label{fig:caso generale}
\end{center}
\end{figure}

The invariant measure induced by this action is nothing but the solid angle on which the point $\underline a_N$ ``sees'' the portions of $S_{N+1}$. The hyperplane through $\underline a_N$ perpendicular to $ao$ intersects $S_{N+1}$ in a sphere $S_N$,
which, by construction, separates $S_{N+1}$ in two portions of equal measure. Therefore, we call it equator. If one excludes to each sphere a small tubular neighbourhood of the equator, with fixed radius $\epsilon$, then the resulting sets have measures
that tends to zero when $N$ diverges. Indeed, this is exactly the same thing which happens to the usual construction for spheres of radius 1 and increasing dimension. The phenomenon is obviously exactly the same, so we can state that the concentration
collapses on the equators. In figure \ref{Misure} a particular case is shown, with two portions that have equal measure with respect to the asymmetric action.

\begin{figure}[!htbp]
\begin{center}
\begin{tikzpicture}[rounded corners]
\draw [fill,black!50!white] (0,-3.9) circle [x radius=0.8pt, y radius=0.5pt];
\draw[ultra thick,black!30!white] (0,-2) circle [x radius=3.4cm, y radius=0.8cm];
\draw [fill,violet] (-4,0) arc (180:360:4cm and 1cm) (4,0) arc (0:180:4cm and 4cm);
\draw [fill, violet] (2.075,-3.41) arc (-58.6793:58.6793-180:4cm and 4cm);
\draw [fill, violet!50!white] (0,-3.2) circle [x radius=2.24 cm, y radius=0.56cm];
\draw[dashed] (0,-3.76) -- (0,-1);
\shade[ball color=red] (0,-2) circle (1.5pt);
\node [red] at (-0.2,-2) {$\pmb a$};
\shade[ball color=blue,opacity=0.5] (0,0) circle (4cm);
\draw [fill] (0,3.9) circle [x radius=0.8pt, y radius=0.5pt];
\draw [thick] (-3.37,-2.1) arc (187.2:352.8:3.4cm and 0.8cm);
\fill [decorate,decoration={text along path,
text=asymmetric equator}]
[fill=blue!20,draw=blue,thick] (-1.5,-2.6) arc (-117:-63:3.4cm and 0.8cm);
\end{tikzpicture}
\caption{The two dark spherical cups have the same volume since are seen by $a$ with the same solid angle. The asymmetric equator separates the sphere in two part having equal volume. The measure tends to concentrate on the asymmetric
equator for the same reason and the same sense the Lebesgue measure concentrates on the equator.}\label{Misure}
\end{center}
\end{figure}

There is, however, an important difference between the asymmetric case and the usual one. The point is that, while the equators remain perfectly symmetric from the point of view of the measures, in the asymmetric case it is not so from the
geometrical perspective. On $S_{N+1}$ we can individuate the south pole as the intersection of the sphere with the half right line from $o$ passing through $\underline a_N$. We see that $\underline a_N$ is the closest to the south pole, with a radius smaller than
that of the equator. \\
Since our sequence of points $\underline a_N$ by construction converges to $a$, with $\|a\|=1$, so that $a\in S_a$, we see that the radii of the equators tend to zero and collapse to $a$. This is shown in figure \ref{concentration}. This difference
is quite important in our opinion for the following reason.\\
In the symmetric case, as we said, the Lebesgue measure concentrates on equators. Which equator depend on choices. Any equator we choose is good, after fixing polar coordinates with azimuth orthogonal to that equatore, everything goes as predicted.
The point is that, however, the concentration does not localize, in the sense that it does not concentrate on a specific point of a limit equator. The same phenomenon happens in the asymmetric case if we consider different equators, see figure
\ref{sconcentration}.

\begin{figure}[!htbp]
\begin{center}
\begin{tikzpicture}[rounded corners]
\draw [fill,black!50!white] (0,-3.9) circle [x radius=0.8pt, y radius=0.5pt];
\draw[ultra thick,black!30!white] (0,-3.8) circle [x radius=0.5cm, y radius=0.18cm];
\draw[red, dashed] (0,-3.8) -- (-1.52,-0.76);
\draw[ultra thick,black!30!white,rotate=26.5651] (0,0) circle [x radius=4cm, y radius=1cm];
\node (a) at (0.1,-3.75) {};
\node [green!80!black] (b) at (5.5,-1.5) {\bf{Equator of} $\pmb {S_{N+1}}$};
\draw[green!80!black,thick] [-latex, bend right] (b) to (a);
\node (f) at (-2.43,-0.305) {};
\node (e) at (-5,0.2) {};
\draw[red,thick] [-latex, bend left] (e) to (f);
\draw[dashed] (0,-3.9) -- (0,3.9);
\draw[dashed] (-8/2.23607,-4/2.23607) -- (8/2.23607,4/2.23607);
\draw [red,fill] (-2.43,-0.305) circle (1pt);
\draw [red,dashed] (-0.44,-1.30) -- (-2.43,-0.305);
\shade[ball color=black!70!white] (0,0) circle (1.5pt);
\shade[ball color=red] (0,-3.8) circle (1.5pt);
\shade[ball color=red] (-1.52,-0.76) circle (1.5pt);
\node at (-0.2,0.1) {$\pmb {o}$};
\node [red] at (-0.22,-3.8) {$\pmb {\underline a_N}$};
\node [red] at (-1.57,-1) {$\pmb {\underline a_{N-1}}$};
\shade[ball color=blue,opacity=0.5] (0,0) circle (4cm);
\draw [fill] (0,3.9) circle [x radius=0.8pt, y radius=0.5pt];
\draw [ultra thick,,rotate=26.5651] (-4,0) arc (180:360:4cm and 1cm);
\draw [red,fill] (-0.44,-1.30) circle (1pt);
\node at (2.5,0) {$\pmb {S_N}$};
\node at (3.5,-2.4) {$\pmb {S_{N+1}}$};
\node (c) at (-0.44,-1.30) {};
\node [red] (d) at (-5,0) {\bf{Equator of} $\pmb {S_N}$};
\draw[red,thick] [-latex, bend right] (d) to (c);
\end{tikzpicture}
\caption{Nearest the point $a$ is to the sphere, smallest is the equator. Since the sequence $\underline a_N$ converges to a point $a$ on $S_a$, the equators of concentration collapse to $a$.}\label{concentration}
\end{center}
\end{figure}

\begin{figure}[!htbp]
\begin{center}
\begin{tikzpicture}[rounded corners]
\draw[ultra thick,black!30!white] (0,-3.8) circle [x radius=0.5cm, y radius=0.18cm];
\draw[ultra thick,black!30!white,rotate=87] (0,0) circle [x radius=3.98cm, y radius=1cm];
\node (a) at (0.3,-3.75) {};
\node [green!80!black] (b) at (6.5,-1.5) {\bf{asymmetric equator}};
\draw[green!80!black,thick] [-latex, bend right] (b) to (a);
\node [blue] (f) at (-0.1,-3.85) {};
\node [blue] (e) at (4,-3.8) {\bf{south pole}};
\draw[blue,thick] [-latex, bend left] (e) to (f);
\draw[dashed] (0,-3.9) -- (0,3.9);
\shade[ball color=black!70!white] (0,0) circle (1.5pt);
\shade[ball color=red] (0,-3.8) circle (1.5pt);
\node at (-0.2,0.1) {$\pmb {o}$};
\node [red] at (-0.22,-3.8) {$\pmb {\underline a_N}$};
\draw [fill,black] (0,-3.9) circle [x radius=0.8pt, y radius=0.5pt];
\shade[ball color=blue,opacity=0.5] (0,0) circle (4cm);
\draw [fill] (0,3.9) circle [x radius=0.8pt, y radius=0.5pt];
\draw [ultra thick,rotate=-93] (-3.98,0) arc (180:360:3.98cm and 1cm);
\node (c) at (-0.95,-1.30) {};
\node [red] (d) at (-5,0) {\bf{meridian equator}};
\draw[red,thick] [-latex, bend right] (d) to (c);
\end{tikzpicture}
\caption{Foliation of the infinite dimensional ball.}\label{sconcentration}
\end{center}
\end{figure}

\newpage

If in place of asymmetric equators we choose meridian equators, as shown in figure \ref{sconcentration}, we will not see a collapse on a point, but a delocalised concentration on the meridians. This is due to the fact these equators do not collapse into a
single point. The concentration on a defined point is due to a ``compactification'' of the sequence of equators as given by Proposition \ref{prop:noncompact}: our deformation is such that the equators relative to the south poles tends to
enter in a compact region around the south pole
$a$. They are flattened to the south pole. In the other direction this does not happen, being as the infinite dimensional sphere is not compact. This (probably) shows the role of compactness in determining the fixed points of the action.
In any case the equators where the concentrations happens are in all cases the images of the concentration locus in the sequence of groups.\\
We will now proceed by further investigating the compact case more deeply.


\subsection{Non symmetric action on a compact space}
Up to now we worked with noncompact spaces. On them we can of course have actions without fixed points. For example, it is sufficient to take $\|a\|<1$ to get an action of $SO(\infty)$ on $S_a$, which has not fixed points. Let us keep in mind
such an action. We want now to extend it to an action on a compact space. To this end, let us first realise the sphere in a well specified space.

Let us consider the real space
\begin{align}
 V:=H^1_0(0,2\pi)=\{f\in W^{1,2}([0,2\pi]) | f(0)=f(2\pi)=0, f \mbox{ absolutely continuous} \}.
\end{align}
The functions
\begin{align}
 u_n(x)=\frac {\sin (nx)}{\sqrt {\pi (n^2+1)}}, \quad n=1,2,3,\ldots
\end{align}
form an orthonormal complete set (CONS) for $V$. The unit sphere $S$ in $V$ can then be described, w.r.t. the given (CONS), as
\begin{align}
 S=\{\sum_{n=1}^\infty x_n u_n | \sum_{n=1}^\infty x_n^2 =1 \}.
\end{align}
Now, consider the embedding $J: H^1_0(0,2\pi) \hookrightarrow L^2([0,2\pi])$. $J$ is a compact operator, which means that the image of a closed bounded set has compact closure. Let us now consider the set $J(S)$.
In $L^2([0,2\pi])$, we consider the real Hilbert subspace $H$ generated by the orthonormal system
\begin{align}
 v_n=\frac {\sin nx}{\sqrt \pi}, \quad n=1,2,\ldots.
\end{align}
Then $J:V\rightarrow H$ is again compact and if $y=J(x)$, we see that its components w.r.t. $\{v_n\}$ are
\begin{align}
 y_n=\frac {x_n}{\sqrt {n^2+1}}.
\end{align}
Therefore, the image $J(S)$ of the unit sphere is
\begin{align}
 J(S)=\{ \sum_{n=1}^\infty y_n v_n| \sum_{n=1}^\infty (n^2+1)y_n^2 =1 \}.
\end{align}
Thus, $J(S)$ looks as an infinite dimensional ellipsoid whose principal form have lengths
\begin{align}
 L_n=\frac 1{\sqrt {n^2+1}}.
\end{align}
In particular, it is squashed more and more in the increasing directions, since $L_n\to0$. This implies that all the points satisfying $\sum_{n=1}^\infty (n^2+1)y_n^2 <1$ are accumulation points for $J(S)$, and we conclude that if
\begin{align}
 B:=\{ x\in V| \|x\|_V \leq 1\},
\end{align}
so that $S=\partial B$, then the compact set generated by $J(S)$ is
\begin{align}
 K_S:= \overline{J(S)}=J(B).
\end{align}
So, in order to get an action of $SO(N)$ on $K_S$ we have to extend the action on $S$ to a continuous action on the whole $B$. To this end, we can foliate $B$ in spheres:
\begin{align}
 B=\bigcup_{\xi\in[0,1]} S_\xi, \quad\ S_\xi=\{x\in V| \|x-a(1-\xi)\|=\xi \}.
\end{align}
If $B_\xi$ is the ball having $S_\xi$ as a boundary, then notice that if $\xi_1<\xi_2$ then $S_{\xi_1}\subset B_{\xi_2}-S_{\xi_2}$. Indeed, for $x\in S_{\xi_1}$ we have
\begin{align}
 \|x-a(1-\xi_2)\|\leq  \|x-a(1-\xi_1)\|+(\xi_2-\xi_1)\|a\|=\xi_1 +(\xi_2-\xi_1)\|a\|=\xi_2 -(\xi_2-\xi_1)(1-\|a\|)<\xi_2,
\end{align}
since we assumes $\|a\|<1$. Therefore, $x\in B_{\xi_2}-S_{\xi_2}$. This allows us to extend the action of the rotations to the whole $B$ as shown in figure \ref{foliation}.

\begin{figure}[!htbp]
\begin{center}
\begin{tikzpicture}[rounded corners]
\draw [blue] (0,-3) -- (0.9,3.9);
\draw [blue] (0,-3) -- (-4,0);
\shade[ball color=black!70!white] (0,0) circle (1.5pt);
\shade[ball color=red] (0,-3) circle (1.5pt);
\node at (-0.2,0.1) {$\pmb {o}$};
\node [red] at (0,-2.8) {$\pmb {a}$};
\draw[ultra thick] (0,-21/8) circle (0.5cm);
\draw[ultra thick] (0,-2) circle (1.3333333cm);
\draw[ultra thick] (0,-1.5) circle (2cm);
\draw[ultra thick] (0,-0.75) circle (3cm);
\draw[ultra thick] (0,0) circle (4cm);
\node at (0.5,-2) {$\pmb {S_{\frac 18}}$};
\node at (1.1,-0.8) {$\pmb {S_{\frac 13}}$};
\node at (1.6,0.2) {$\pmb {S_{\frac 12}}$};
\node at (2.3,1.6) {$\pmb {S_{\frac 34}}$};
\node at (3,3) {$\pmb {S}$};
\shade[ball color=blue] (0.9,3.9) circle (1.5pt);
\shade[ball color=blue] (-4,0) circle (1.5pt);
\shade[ball color=blue] (0.675,2.173) circle (1.5pt);
\shade[ball color=blue] (0.45,0.45) circle (1.5pt);
\shade[ball color=blue] (0.3,-0.7) circle (1.5pt);
\shade[ball color=blue] (-3,-0.75) circle (1.5pt);
\shade[ball color=blue] (-2,-1.5) circle (1.5pt);
\shade[ball color=blue] (-1.333,-2) circle (1.5pt);
\shade[ball color=blue] (0.11,-2.131) circle (1.5pt);
\shade[ball color=blue] (-0.5,-21/8) circle (1.5pt);
\draw [-latex, thick, red] (-1.12,-2.16) arc (180-36.8699:82.5686:1.4cm);
\node [red] at (-0.6,-1.5) {$\pmb {\theta}$};
\node [blue] at (-4.2,0) {$\pmb {p}$};
\node [blue] at (-3.2,-0.75) {$\pmb {q}$};
\node [blue] at (-2.2,-1.5) {$\pmb {r}$};
\node [blue] at (-1.533,-2) {$\pmb {s}$};
\node [blue] at (-0.7,-21/8) {$\pmb {t}$};
\node [blue] at (0.9,4.2) {$\pmb {R_\theta(p)}$};
\node [blue] at (0.68,2.4) {$\pmb {R_\theta(q)}$};
\node [blue] at (0.45,0.7) {$\pmb {R_\theta(r)}$};
\node [blue] at (0.4,-0.45) {$\pmb {R_\theta(s)}$};
\node [blue] at (-0.1,-1.9) {$\pmb {R_\theta(t)}$};
\end{tikzpicture}
\caption{Asymmetric action of rotations on $B$ through the foliation.}\label{foliation}
\end{center}
\end{figure}

At this point we are ready to discuss how a measure on $K_S$ is induced by an ``observer''. Let us  fix any point of $K_S$, say $s$ in figure \ref{foliation} for example. This is our ``observer''. If we consider the orbit of $SO(\infty)$ through $s$, we get a
measure on $K_S$ as defined above. In particular, we define the the family of measures $\mu_N^{s}$, defined as the normalized angular measure under the action of $SO(N+1)$. It is defined as follows. Suppose $M$ is the lower integer such that
$(u_M|s-a)\neq 0$. Then, for $N<M$, $s$ is left fixed by the action of $SO(N+1)$, so that $\mu_N^{s}=D_s$ is the Dirac measure with support $J(s)$. If $N\geq M$, let $\mathcal O_N^s$ the orbit of $SO(N+1)$ through $s$. Then, for
any measurable subset of $K_S$ (a subset $E\subset K_S$ such that $E\cap \mathbb R^n$ is Lebesgue measurable for any $n$), we set
\begin{align}
 \mu_N^{s} (E):= \Omega_N (\mathcal O^s_N \cap E),
\end{align}
where $\Omega_N(R)$ is the solid angle in $\mathbb R^{N+1}$ under which $J_{a}$ sees $R\in \mathbb R^{N+1}$. Here we mean
\begin{align}
 \mathbb R^{N+1}=J(a)+\mathbb R v_1+\ldots+\mathbb R v_{N+1}\subset L^2([0,2\pi]).
\end{align}
Also, we mean the angular measure to be normalized so that $\Omega_N (\mathcal O^s_N)=1$. We now show by hand that the measures so defined on the compact concentrate on $a$.
\begin{prop}
 Fix arbitrarily $\varepsilon>0$ and consider $U_\varepsilon(a)$ the open ball neighbourhood of radius $\varepsilon$ of $J(a)$ in $L^2([0,2\pi])$. If $E$ is a measurable set such that $E\subseteq K_S\cap U_\varepsilon(a)^c$, then
\begin{align}
 \lim_{N\to\infty} \mu_N^s(E)=0.
\end{align}
\end{prop}
\begin{proof}
 It is sufficient to prove that $ \lim_{N\to\infty} \mu_N^s(U_\varepsilon(a)^c)=0.$ Since $K_S$ is an ellipsoid whose half diameter in the $(N+1)$-th direction is $L_{N+1}=((N+1)^2+1)^{\frac 12}$, we see that for $N$ large enough we have
 $L_{N+1}<\varepsilon$. If $\theta_{N+1}$ is the angle in $\mathbb R^{N+1}$ w.r.t. the axis fixed by $v_{N+1}$, we get that $U_\varepsilon(a)^c$ is seen by $J(a)$ under an angle range $[\pi_2-\bar \theta, \pi_2+\bar \theta]$, with
\begin{align}
 \sin \theta \leq \frac {L_{N+1}}{\varepsilon}.
\end{align}
Thus, it looks is seen to be around an equator by an angle $\bar \theta$ that goes to zero as
\begin{align}
 \frac {L_{N+1}}{\varepsilon} \sim \frac 1{\varepsilon N}
\end{align}
as $N$ diverges. Since it decreases faster than $N^{-\frac 12}$, the thesis follows immediately from Levy's argument.
\end{proof}
\begin{rem}\rm
It follows that the measures concentrate uniformly on the fixed point $J(a)$, contrarily to the non compact case. Once again this is a consequence of the fact that as $N$ grows, the concentration locus of $SO(N)$ is mapped into a smaller neighbourhood
of the point $a$.
\end{rem}

\

  From our attempts it seems that even more interesting examples could arise from non linear actions of Lie groups.
Unfortunately, analytical non linear Lie group actions are very difficult to determine. The most notable examples can be found out in algebraic topology whenever the orbit space is investigated.
Regrettably, in these cases only the nature of the action is determined rather than their actual analytic form.
Indeed, they set algebraic criteria to detect whether an action is equivalent to a linear one or not (see for example \cite{Br,Br1},\cite{HsHs}). This does not serve our scopes.
However our results in this area are still not satisfactory enough, so we have decided to postpone this topical discussion to future research.


\section{Applications in Physics}
Finally, we apply our approach to some examples from physics.
The phenomenon of concentration of measure is of course relevant in describing statistical ensembles.
The measure on the phase space determines the equilibrium configuration of the statistical system.
The eventual concentration of measure is therefore characterizing for the system.
As an application of what we have seen up to now, let us consider a statistical ensemble of free particles on $SU(N)$ manifolds.
The phase space is given by the cotangent bundle

\begin{align}
 T^*SU(N)\simeq SU(N)\times \mathfrak {su}(N)^*.
\end{align}

Since $SU(N)$ is compact, it is endowed with a negative definite Killing product that naturally induces a negative definite scalar product over $\mathfrak {su}(N)^*$:
\begin{align}
 \langle,\rangle: \mathfrak {su}(N)^*\times \mathfrak {su}(N)^*\longrightarrow \mathbb R.
\end{align}
It defines the kinetic energy of the particles of mass $m$ by
\begin{align}
E(\vec p)=-\frac {\langle \vec p, \vec p\rangle}{2m}, \qquad\ \vec p\in \mathfrak {su}(N)^*.
\end{align}
If $\kappa$ is the Boltzmann constant and $dm_N$ the Lebesgue measure over $\mathfrak {su}(N)^*$, the Gibbs measure at temperature $T$ is therefore
\begin{align}
d\mu(\xi, \vec p)=\frac 1{\mathcal Z_N} e^{-\frac {\langle \vec p, \vec p\rangle}{2m\kappa T}} \ d\mu_{SU(N)}(\xi)\ dm_N (\vec p), \qquad (\xi,\vec p)\in T^*SU(N),
\end{align}
where $\mathcal Z$ is the normalization factor. We therefore get an easy result for the concentration of the measure here, since the Gaussian measure concentration is discussed in \cite{Le}, page 2. This distribution is uniform, which means that the concentration will
not appear to be on an ``a priori defined locus'', but that the locus depends on the observer. This is the same as expected in a homogeneous distribution on a uniform sphere of large dimension: any given observer will conclude that the largest probability
is in finding a particle near the equator with respect to the observer placed in the north pole. \\
In the example we have just considered the measure concentrates without fixed points. It is interesting to consider the same situation in an example where the concentration determines a fixed point. An example is the case of the spheres when the
measure is determined by the asymmetric action, as presented in Section \ref{Examples}. Again, we can consider a free gas of particles of mass $m$ on $S_{N+1}$, so that the Gibbs measure over $T^*S_{N+1}$ is
\begin{align}
 d\mu_{Gibbs} (x,p)=d\mu_{\underline a_N} (x) dm_{N+1}(p) e^{-\frac {p^2}{2m \kappa T}},
\end{align}
where $m_{N+1}$ is the Lebesgue measure on $\mathbb R^{N+1}$. The consequence of our analysis in sections \ref{5.4} and \ref{pictorialx} is that the measure concentrates on neighbourhoods of about $1/\sqrt N$ of radius around asymmetric equators
collapsing over $a$, while momenta behaves exactly as in the previous example. Therefore, this looks as if the equators, and the fixed point behave as attractors for the particles. This could be how the concentration seems to contribute to
the gravitational effects.


\section*{Acknowledgments}
We would like to thank Riccardo Re and Vladimir Pestov for several valuable discussions, and Vittorino Pata for some very helpful suggestions.


\end{document}